\documentclass[letterpaper, 10 pt, conference]{ieeeconf}
\IEEEoverridecommandlockouts                              % This command is only needed if 
\newtheorem{theorem}{Theorem}

\newtheorem{lemma}[theorem]{Lemma}
\newtheorem{proposition}[theorem]{Proposition}

\newcommand{\non}{\nonumber}

\newcommand{\beq}{\begin{equation}}
\newcommand{\eeq}{\end{equation}}
\newcommand{\bseq}{\begin{subequations}}
\newcommand{\eseq}{\end{subequations}}
\newcommand{\beqn}{\begin{eqnarray}}
\newcommand{\eeqn}{\end{eqnarray}}
\newcommand{\ba}{\begin{array}}
\newcommand{\ea}{\end{array}}
\newcommand{\bct}{\begin{center}}
\newcommand{\ect}{\end{center}}
\newcommand{\btmz}{\begin{itemize}}
\newcommand{\etmz}{\end{itemize}}
\newcommand{\benum}{\begin{enumerate}}
\newcommand{\eenum}{\end{enumerate}}

\newcommand{\matbegin}{
        \left[
}
\newcommand{\matend}{
        \right]
}

\newcommand{\tbo}[2]{
  \matbegin \begin{array}{c}
       #1 \\ #2
       \end{array} \matend }

\newcommand{\tbt}[4]{
  \matbegin \begin{array}{cc}
       #1 & #2 \\ #3 & #4
       \end{array} \matend }

\newcommand{\be}{\begin{equation}}
\newcommand{\ee}{\end{equation}}

\newcommand{\cplxs}{ C\kern -.35em \rule{0.03 em}{.7 ex}~   }

\def\complex{\hbox{C\kern -.45em \rule{0.03 em}{1.5 ex}}~}

\newcommand{\bi}{\begin{itemize}}
\newcommand{\ei}{\end{itemize}}

\newtheorem{assumption}{Assumption}

\usepackage{amssymb,amsmath,latexsym,color,amsmath,pifont,epsfig,graphicx}
\usepackage{setspace,cite}
\usepackage{algorithm}
\usepackage{algorithmic}
\usepackage{subfig}  %Use as needed
\usepackage{algorithm}
\usepackage{algorithmic}
\usepackage{soul}
\usepackage{graphics}
\usepackage{multirow}
\usepackage{bm}
\usepackage{rotating}
\usepackage{epsfig}
\usepackage{amssymb,latexsym,color,amsmath,pifont,epsfig,mathtools}
\usepackage{graphicx,booktabs}
\usepackage{indentfirst}
\usepackage{dsfont}
\usepackage{url}
\usepackage{color}
\usepackage{url}
\usepackage{multicol}
\usepackage{tablefootnote}

\newcommand{\vsp}{\vspace*{0.cm}}
\newcommand{\prox}{\mathbf{prox}}

\DeclareMathOperator*{\argmin}{argmin}
\DeclareMathOperator*{\minimize}{minimize}
\DeclareMathOperator*{\subject}{subject~to}

\newcommand{\mre}{\mathrm{e}}

\newcommand{\sign}{\mathrm{sign}}

\newcommand{\DefinedAs}[0]{\mathrel{\mathop:}=}

\begin{document}
\title{Global exponential stability of primal-dual gradient flow dynamics based on the proximal augmented Lagrangian: A Lyapunov-based approach}

\author{Dongsheng~Ding,~\IEEEmembership{Student~Member,~IEEE,}
        and Mihailo~R.~Jovanovi\'{c},~\IEEEmembership{Fellow,~IEEE}
\thanks{Financial support from the National Science Foundation under awards ECCS-1708906 and ECCS-1809833 is gratefully acknowledged.}% <-this % stops a space
\thanks{D.\ Ding and M.\ R.\ Jovanovi\'{c} are with the Ming Hsieh Department of Electrical and Computer Engineering, University of Southern California, Los Angeles, CA 90089.
		{E-mails: dongshed@usc.edu, mihailo@usc.edu.}
	    }% <-this % stops a space
%\thanks{Manuscript received April 1, 2019; revised August 1, 2019.}
        }

\maketitle

	\begin{abstract}
	For a class of nonsmooth composite optimization problems with linear equality constraints, we utilize a Lyapunov-based approach to establish the global exponential stability of the primal-dual gradient flow dynamics based on the proximal augmented Lagrangian. The result holds when the differentiable part of the objective function is strongly convex with a Lipschitz continuous gradient; the non-differentiable part is proper, lower semi-continuous, and convex; and the matrix in the linear constraint is full row rank. Our quadratic Lyapunov function generalizes recent result from strongly convex problems with either affine equality or inequality constraints to a broader class of composite optimization problems with nonsmooth regularizers and it provides a worst-case lower bound of the exponential decay rate. Finally, we use computational experiments to demonstrate that our convergence rate estimate is less conservative than the existing alternatives.
	\end{abstract}

	\vspace*{-1ex}
\section{Introduction}

Primal-dual gradient flow dynamics belong to a class of Lagrangian-based methods for constrained optimization problems. Among other applications, such dynamics have found use in network utility maximization~\cite{feipag10}, resource allocation~\cite{dinjovACC18}, distributed optimization~\cite{wannic11}, and feedback-based online optimization~\cite{coldalber19} problems. Stability conditions for various forms of the gradient flow dynamics have been proposed since their introduction in the 1950's~\cite{arrhuruza58}. 

	% \vsp

Lyapunov-based approach has been an effective tool for studying the stability of primal-dual algorithms starting with the seminal paper of Arrow, Hurwicz, and Uzawa~\cite{arrhuruza58}. They utilized a quadratic Lyapunov function to establish the global asymptotic stability of the primal-dual dynamics for strictly convex-concave Lagrangians. This early result was extended to the problems in which the Lagrangian is either strictly convex or strictly concave~\cite{palchi06}. A simplified Lyapunov function was also proposed for linearly-convex or linearly-concave Lagrangians. For a projected variant of the primal-dual dynamics that could account for inequality constraints, a Krasovskii-based Lyapunov function was combined with LaSalle's invariance principle to show the global asymptotic stability in~\cite{feipag10}. The invariance principle was also specialized to discontinuous Carath\'eodory systems and a quadratic Lyapunov function was used to show the global asymptotic stability of projected primal-dual gradient flow dynamics under globally-strict or locally-strong convexity-concavity assumptions~\cite{chemalcor16,chemalcor18}. Additional information about the utility of a Lyapunov-based analysis in optimization can be found in a recent reference~\cite{polshc17}.

	% \vsp

In~\cite{dhikhojovTAC19}, the theory of proximal operators was combined with the augmented Lagrangian approach to solve optimization problems in which the objective function can be decomposed into the sum of the strongly convex term with a Lipschitz continuous gradient and a convex non-differentiable term. By evaluating the augmented Lagrangian along a certain manifold, a continuously differentiable function of both primal and dual variables was obtained. This function was named the proximal augmented Lagrangian and the theory of integral quadratic constraints (IQCs) in the frequency domain was employed to prove the global exponential stability of the resulting primal-dual dynamics~\cite{dhikhojovTAC19}. This method yields an evolution model with a continuous right-hand-side even for nonsmooth problems and it avoids the explicit construction of a Lyapunov function. In~\cite{quli19}, a quadratic Lyapunov function was used to prove similar properties for a narrower class of problems that involve strongly convex and smooth objective functions with either affine equality or inequality constraints. More recently, this Lyapunov-based result was extended to account for variations in the constraints~\cite{tanquli19,cheli19} and the theory of IQCs was used to prove global exponential stability of the differential equations that govern the evolution of proximal gradient and Douglas-Rachford splitting flows~\cite{mogjovAUT19}.    

Herein, we utilize a Lyapunov-based approach to establish the global exponential stability of the primal-dual gradient flow dynamics resulting from the proximal augmented Lagrangian. As aforementioned, this method was introduced in~\cite{dhikhojovTAC19} to solve a class of nonsmooth composite optimization problems. When the differentiable part of the objective function is strongly convex with a Lipschitz continuous gradient; the non-differentiable part is proper, lower semi-continuous, and convex; and the matrix in the linear constraint is full row rank, we construct a new quadratic Lyapunov function for the underlying primal-dual dynamics. This Lyapunov function allows us to derive a worst-case lower bound on the exponential decay rate. In contrast to~\cite{feipag10,niecor16,chemalcor16,chemalcor18}, our gradient flow dynamics are projection-free and there are no nonsmooth terms in the Lyapunov function. 

We employ the theory of IQCs in the time domain to obtain a quadratic Lyapunov function that establishes the global exponential stability of the primal-dual gradient flow dynamics resulting from the proximal augmented Lagrangian framework for nonsmooth composite optimization. Our Lyapunov function is more general that the one in~\cite{quli19} and it yields less conservative convergence rate estimates. This extends the results of~\cite{quli19} from strongly convex problems with affine equality or inequality constraints to a broader class of optimization problems with nonsmooth regularizers. 
	
The remainder of the paper is organized as follows. In Section~\ref{sec.pro}, we provide background material, formulate the nonsmooth composite optimization problem, and describe the proximal augmented Lagrangian as well as the resulting primal-dual gradient flow dynamics. In Section~\ref{sec.glob}, we construct a quadratic Lyapunov function for verifying the global exponential stability of the primal-dual dynamics. In Section~\ref{sec.comp}, we use computational experiments to illustrate the utility of our results. In Section~\ref{sec.con}, we close the paper with concluding remarks.

	\vspace*{-1ex}
\section{Problem formulation and background}
\label{sec.pro}

We consider convex composite optimization problems in which the objective function consists of a continuously differentiable term $f$ and a non-differentiable term $g$
\be
\label{eq.opt}
\ba{rl}
\minimize\limits_{x,\,z}
&
f(x) \,+\, g(z)\; 
\\[0.15cm]
\subject 
& 
Tx \, - \, z \, = \, 0
\ea
\ee
where $T\in \mathbb{R}^{m\times n}$ is a matrix that relates the optimization variables $x\in\mathbb{R}^n$ and $z\in\mathbb{R}^m$. 

	\vsp
	
\begin{assumption}
	\label{as.feasible}
Problem~\eqref{eq.opt} is feasible and its minimum is finite. 	
\end{assumption}	
\vsp

\begin{assumption}
	\label{as.fg}
	The continuously differentiable function $f$ is $m_f$-strongly convex with an $L_f$-Lipschitz continuous gradient and the non-differentiable function $g$ is proper, lower semi-continuous, and convex.
	\end{assumption}

\vsp

\begin{assumption}
	\label{as.tt}
	The matrix $T \in \mathbb{R}^{m\times n}$ has a full row rank. 
\end{assumption}

	 \vspace*{-1ex}
\subsection{Proximal augmented Lagrangian}

The proximal operator of the function $g$ is given by~\cite{parboy14}
\[
\prox_{\mu g}(v) \, \DefinedAs\, \argmin_x\, \left( g(x)\,+\,\frac{1}{2\mu}\Vert x \,-\, v \Vert^2 \right)
\]
and the associated value function is Moreau envelope,
\[
M_{\mu g}(v) 
\, \DefinedAs \, 
g(\prox_{\mu g}(v)) \, + \, \frac{1}{2\mu} \, \|\prox_{\mu g}(v) \, - \, v\|^2
\]
where $\mu$ is a positive parameter. The Moreau envelope is continuously differentiable, even when $g$ is not, and its gradient is determined by,
\be
\nabla M_{\mu g}(v)
\,=\,
\dfrac{1}{\mu} \left( v \,-\, \prox_{\mu g}(v) \right).
%\label{eq.gradM}
\non
\ee

The augmented Lagrangian of the constrained optimization problem~\eqref{eq.opt} is given by
\be
\non
\mathcal{L}(x,z;y)
\,=\,
f(x)\,+\,g(z)\,+\,y^T(Tx-z)\,+\,\dfrac{1}{2\mu}\|Tx-z\|^2
\ee
where $x\in\mathbb{R}^n$ and $z\in\mathbb{R}^m$ are the primal variables, $y \in \mathbb{R}^m$ is a dual variable, and $\mu$ is a positive parameter. Completion of squares brings~$\mathcal{L}(x,z;y)$ into the following form 
\[
\mathcal{L}(x,z;y)
\,=\,
f(x)\,+\,g(z)\,+\,\frac{1}{2\mu}\|z-(Tx+\mu y)\|^2\,-\,\frac{\mu}{2}\|y\|^2.
\]
The minimizer of the augmented Lagrangian with respect to $z$ is 
\[
z_\mu^\star (x; y)
\, = \,
\prox_{\mu g}( Tx \, + \, \mu y)
\]
where $\prox_{\mu g}$ denotes the proximal operator of the function $g$. Restriction of $\mathcal{L}$ along the manifold determined by $z_\mu^\star (x; y)$ yields the proximal augmented Lagrangian~\cite{dhikhojovTAC19},
\be
\ba{rrl}
\mathcal{L}_\mu(x; y)
& \!\!\! \DefinedAs \!\!\! &
\mathcal{L}(x,z_\mu^\star (x; y);y)
\\[0.1cm]
& \!\!\! = \!\!\! &
f(x)\,+\,M_{\mu g}(Tx+\mu y) \, - \, \dfrac{\mu}{2}\left\|y\right\|^2
\ea
\label{eq.pal}
\ee
where $M_{\mu g}$ is the Moreau envelope of the function $g$. Continuous differentiability of the proximal augmented Lagrangian $\mathcal{L}_\mu(x; y)$ with respect to both $x$ and $y$ follows from continuous differentiability of $M_{\mu g}$ and Lipschitz continuity of the gradient of $f$. 

\vspace*{-1ex}
\subsection{Examples}

We next provide examples of convex optimization problems that can be brought into the form~\eqref{eq.opt}. For instance, the problem with linear equality constraints,
\begin{equation}
\begin{array}{rl}
\minimize\limits_{x}
&
f(x)\; 
\\[0.15cm]
\subject 
& 
Tx \, = \, b
\end{array}
\label{eq.ex3}
\end{equation}
where $b \in \mathbb{R}^{m}$ is a given vector can be cast as~\eqref{eq.opt} by choosing $g(z)$ to be an indicator function, $g(z_i) \DefinedAs \{ 0, \, z_i = b_i$; $\infty, \, \mbox{otherwise} \}$. In this case, the proximal operator is given by $\prox_{\mu g} (v_i) = b_i $, the associated Moreau envelope is $M_{\mu g}(v_i) = \frac{1}{2\mu} \, (v_i-b_i)^2$, and the gradient of the Moreau envelope is 
$
\nabla M_{\mu g}(v_i)
= 
(v_i-b_i)/\mu.
$

	\vsp
	
The problem with linear inequality constraints,
\begin{equation}
\begin{array}{rl}
\minimize\limits_{x}
&
f(x)\; 
\\[0.15cm]
\subject 
& 
Tx \, \leq \, b
\end{array}
\label{eq.ex1}
\end{equation}
where $b \in \mathbb{R}^{m}$ is a given vector can be cast as~\eqref{eq.opt} by choosing $g(z)$ to be an indicator function, $g(z_i) \DefinedAs \{ 0, \, z_i \leq b_i$; $\infty, \, \mbox{otherwise} \}$. The proximal operator is $\prox_{\mu g} (v_i) = \min\{v_i, b_i\} $, the associated Moreau envelope is $M_{\mu g}(v_i) = \{\frac{1}{2\mu} \, (v_i-b_i)^2, v_i> b_i; 0, \mbox{otherwise}\}$, and the gradient of the Moreau envelope is 
$
\nabla M_{\mu g}(v_i)
= 
\max (0,(v_i-b_i)/\mu).
$

\vsp

Unconstrained optimization problems with nonsmooth regularizers can be also represented by~\eqref{eq.opt}. 
For example, the logistic regression with elastic net regularization~\cite{zouhas05}
\be
\label{eq.ex2}
\ba{rl}
\minimize\limits_{x}
&
\ell(x) \,+\, \frac{1}{2} \|x\| ^2\,+\, \|x\|_1
\ea
\ee
where the logistic loss $\ell(x)$ is given by
$
\sum_{i = 1}^{d} ( \log(1 + \mre^{a_i^T x}) - y_i a_i^T x )
$
where $a_i$ is the feature vector and $y_i\in\{0,1\}$ is the corresponding label. Choosing $f(x) \DefinedAs \ell(x) + \frac{1}{2} \|x\| ^2 $, $g(z) \DefinedAs \|z\|_1 $, and $T \DefinedAs I$ brings~\eqref{eq.ex2} into~\eqref{eq.opt}. The proximal operator is the soft-thresholding $\prox_{\mu g} (v_i) = \text{sign}(v_i)\max\{\vert v_i\vert-\mu, 0\} $, the associated Moreau envelope is the Huber function $M_{\mu g}(v_i) = \{\frac{1}{2\mu} \, v_i^2, \vert v_i\vert\leq \mu; \vert v_i\vert-\frac{\mu}{2},\vert v_i\vert \geq \mu \}$, and the gradient of the Moreau envelope is the saturation function
	$
	\nabla M_{\mu g}(v_i)
	=
	\sign(v_i) \min \, ( |v_i|/\mu, \,1 ).
	$
	
	\vspace*{-1ex}
\subsection{Primal-dual gradient flow dynamics}
The primal-dual gradient flow dynamics can be used to compute the saddle points of~\eqref{eq.pal}, 
	\begin{subequations}
	\label{eq.ahu}
\be
\dot{w} 
\, = \,
F(w)
\ee
where $w \DefinedAs [ \, x^T \; y^T \, ]^T$ and 
\be
\ba{rrl}
F(w) 
& \!\! \DefinedAs \!\! &	
\tbo{- \nabla_x \mathcal{L}_{\mu}(x;y)}{\phantom{-} \nabla_y \mathcal{L}_{\mu}(x;y)} 
\\[0.35cm]
& \!\! = \!\! &
\tbo{
	- ( \nabla f(x) + T^T \nabla M_{\mu g}(Tx+\mu y))
}
{
	\mu \, ( \nabla M_{\mu g}(Tx+\mu y) \, - \, y )
}.
\ea
\ee
\end{subequations}
Let $\bar{w} \DefinedAs [ \, \bar{x}^T \, \bar{y}^T \, ]^T$ denote the equilibrium points of~\eqref{eq.ahu}, i.e., the solutions to $F(\bar{w}) = 0$. The Lagrangian of the optimization problem~\eqref{eq.opt} is given by $f(x)+g(z)+y^T(Tx-z)$ and the associated KKT optimality condition are, 
\be
\ba{rcl}
0 & \!\! = \!\! & \nabla f(x^\star) \,+\, T^T y^\star
\\[0.1cm]
0 & \!\!\in \!\! & \partial g(z^\star) \,-\, y^\star
\\[0.1cm]
0 & \!\! = \!\! & Tx^\star \,-\, z^\star
\ea
\label{eq.kkt}
\ee
where $\partial g$ is the subgradient of $g$. The following lemma establishes the relation between $\bar{w}$ and the optimality conditions~\eqref{eq.kkt}; see~\cite{dhikhojovTAC19} for details.

\vsp

\begin{lemma}\label{lem.sol}
	Let Assumptions~\ref{as.feasible} and~\ref{as.fg} hold. The equilibrium point $\bar{w} \DefinedAs [ \, \bar{x}^T \, \bar{y}^T \, ]^T$ of the primal-dual gradient flow dynamics~\eqref{eq.ahu} satisfies optimality conditions~\eqref{eq.kkt} with $\bar{z} \DefinedAs \prox_{\mu g} (T\bar{x} + \mu \bar{y})$. Moreover, $(\bar{x},\bar{z} )$ is the optimal solution of nonsmooth composite optimization problem~\eqref{eq.opt}. 
\end{lemma}

\vsp

Under Assumptions~\ref{as.feasible}-\ref{as.tt}, the global exponential stability of the primal-dual gradient flow dynamics~\eqref{eq.ahu} was established in~\cite{dhikhojovTAC19} by employing the theory of integral quadratic constraints in the frequency domain. An upper bound on the convergence rate was also obtained but the explicit form for the quadratic Lyapunov function was not provided. Recent reference~\cite{quli19} used a Lyapunov-based approach to show the global exponential stability for a class of problems with a strongly convex and smooth objective function $f$ subject to either affine equality or inequality constraints. In our preliminary work~\cite{dinjovACC19}, a similar quadratic Lyapunov function was used to prove global exponential stability of the primal-dual gradient flow dynamics~\eqref{eq.ahu}. In what follows, we employ the theory of IQCs in the time domain to obtain a quadratic Lyapunov function that establishes the global exponential stability of~\eqref{eq.ahu} and yields less conservative convergence rate estimates.

	\vspace*{-1ex}
\section{Global exponential stability via quadratic Lyapunov function}
	\label{sec.glob}

In this section, we identify a quadratic Lyapunov function that can be used to establish the global exponential stability of the primal-dual gradient flow dynamics~\eqref{eq.ahu} for strongly convex problems~\eqref{eq.opt} with full row rank matrix $T$ and provide an estimate of the convergence rate.  

\begin{figure}
	\centering
	\includegraphics[width=2.75in,height=1.25in]{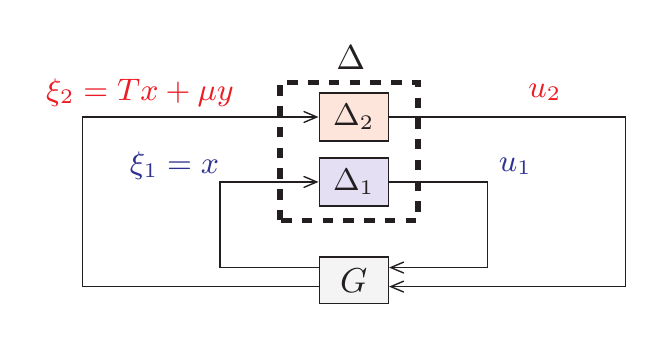}
	\caption{Block diagram of primal-dual gradient flow dynamics~\eqref{eq.ahu}: $G$ is an exponentially stable LTI system in~\eqref{eq.dtdyn1} and $\Delta$ is a static nonlinear map that satisfies quadratic constraint~\eqref{eq.IQC}. } 
	\label{fig.block}
\end{figure}

	\vspace*{-1ex}
\subsection{A system-theoretic viewpoint of primal-dual dynamics}

Inspired by~\cite{lesrecpac16}, we view~\eqref{eq.ahu} as a feedback interconnection of an LTI system with static nonlinearities; see Fig.~\ref{fig.block}. These are determined by the gradient of the smooth part of the objective function $\nabla f$ and the proximal operator $\prox_{\mu g}$. Structural properties of nonlinear terms that we exploit in our analysis are specified in Assumption~\ref{as.fg}. 

Let $u = [ \, u_1^T \; u_2^T \, ]^T$ and $\xi = [ \, \xi_1^T \; \xi_2^T \, ]^T$, with
\be
\ba{rcl}
\xi_1 & \!\! \DefinedAs \!\! &x
\\[0.1cm]
\xi_2 & \!\! \DefinedAs \!\! & Tx+\mu y
\\[0.1cm]
u_1 & \!\!\DefinedAs \!\! & \Delta_1(\xi_1) \; = \; \nabla f(x)-m_f x
\\[0.1cm]
u_2 & \!\! \DefinedAs \!\! & \Delta_2(\xi_2) \; = \; \prox_{\mu g} (Tx+\mu y).
\ea
\non
\ee
For strongly convex $f$, the primal-dual dynamics~\eqref{eq.ahu} can be cast as an LTI system $G$ in feedback with a nonlinear block $\Delta$, where
	\begin{subequations}
	\label{eq.dtdyn}
\be
\ba{rcl}
\dot{w} 
& \!\!\! = \!\!\! & A w \,+\, B u
\\[0.1cm] 
\xi
& \!\!\! = \!\!\! &  C w
\\[0.1cm] 
u
& \!\!\! = \!\!\! &  \Delta (\xi)
\ea
\label{eq.dtdyn1}
\ee
with
\be
\begin{aligned}
	& A = \tbt{
		-(m_f I + \tfrac{1}{\mu} T^T T) 
	}{-T^T}{T}{ 0 }
	\\
	& B =\tbt{-I}{\tfrac{1}{\mu}T^T}{0}{- I},
	\;\;
	C = \tbt{I}{0}{T}{\mu I}.
\end{aligned}
\label{eq.ABC}
\ee
\end{subequations}
The input is given by $u = \Delta(\xi)$ where $\Delta (\xi)$ is a $2\times2$ block-diagonal matrix with the diagonal blocks $\Delta_1 (\xi_1)$ and $\Delta_2 (\xi_2)$. These nonlinearities satisfy the pointwise quadratic inequalities~\cite{dhikhojovTAC19}
	\be
\begin{bmatrix}
	\xi_i \, - \, \bar{\xi}_i\\
	u_i \, - \, \bar{u}_i \\
\end{bmatrix}^T\tbt {0}{\hat{L}_i I}{\hat{L}_i I}{-2I}
\begin{bmatrix}
	\xi_i \, - \, \bar{\xi}_i\\
	u_i \, - \, \bar{u}_i
\end{bmatrix} \,\geq\, 0
\non
\ee
where  
	$\bar{w}
\DefinedAs
[\,
\bar{x}^T \, \bar{y}^T
\,]^T$ is the equilibrium point of system~\eqref{eq.dtdyn}, $\bar{\xi}_1=\bar{x}$, $\bar{\xi}_2=T\bar{x}+\mu\bar{y}$, $\hat{L}_1 \DefinedAs L_f - m_f$, and $\hat{L}_2 = 1$. This is because $\Delta_1$ is the gradient of the convex function $f(\xi_1)-(m_f/2)\|\xi_1\|^2$ and, thus, it is Lipschitz continuous with parameter $L_f - m_f$~\cite[Proposition~5]{lesrecpac16}; and $\Delta_2$ is given by the proximal operator of the function $g$ and, thus, it is firmly non-expansive (i.e., Lipschitz continuous with parameter one)~\cite{parboy14}. These quadratic constraints can be combined into,
	\be
\tbo{\xi \, - \, \bar{\xi}}{u \, - \, \bar{u}}^T
\underbrace{\tbt {0}{\Pi_0}{\Pi_0}{-2 \Lambda}}_{\Pi}
\tbo{\xi \, - \, \bar{\xi}}{u \, - \, \bar{u}}
\,\geq\, 0
\label{eq.IQC}
\ee
where
\[
\Pi_0
\; = \;
\tbt{\lambda_1 \hat{L}_1 I}{0}{0}{\lambda_2 I},
~~
\Lambda
\; = \;
\tbt{\lambda_1 I}{0}{0}{\lambda_2 I}
\]
and $\lambda_1$, $\lambda_2$ are non-negative scalars.

	\vspace*{-1ex}
\subsection{Lyapunov-based analysis for global exponential stability}

For the primal-dual gradient flow dynamics~\eqref{eq.ahu} with equilibrium point $\bar{w}$, we propose a quadratic Lyapunov function candidate
	\begin{subequations}
	\label{eq.lyap}
	\be
	V(\tilde{w}) 
	\; = \; 
	\tilde{w}^T P \, \tilde{w}
	\ee 
with $\tilde{w} \DefinedAs w - \bar{w}$ and
	\be
	P 
	\, = \,  
	\alpha
	\left[
	\begin{array}{cc}
	I & (1/\mu) \, T^T
	\\
	 (1/\mu) \, T & (1 + {m_f}/{\mu}) \, I \, + \, (1/{\mu^2}) \, TT^T
	\end{array}\right]
	\label{eq.P}
	\ee
	\end{subequations}
where $\alpha$ is a positive parameter, $m_f$ is the strong convexity module of the function $f$, $\mu$ is the augmented Lagrangian parameter, and $T$ is the full rank matrix associated with the linear equality constraint in~\eqref{eq.opt}. The matrix $P$ is positive definite and, for $A$ in~\eqref{eq.ABC}, we have  
	\be
	A^T P \, + \, PA 
	\,=\,
	-
	2 \alpha
	\left[
	\begin{array}{cc}
		m_f I & 0
		\\
		0 & (1/\mu) \, TT^T
	\end{array}\right]
	\, \prec \, 0.
	\label{eq.PAAP}
	\ee
Thus, $A$ is a Hurwitz matrix and the LTI system in Fig.~\ref{fig.block} is exponentially stable. Furthermore, the derivative of $V$ along the solutions of~\eqref{eq.dtdyn1} is determined by
	\begin{subequations}
	\be
	\dot{V}
	\; = \; 
	\tbo{\tilde{w}}{\tilde{u}}^T
	\tbt{A^T P + P A}{P B}{B^T P}{0}
	\tbo{\tilde{w}}{\tilde{u}}
	\label{eq.Vdot}
	\ee
where $\tilde{u} \DefinedAs u - \bar{u}$, and the substitution of the output equation $\xi = C w$ in~\eqref{eq.dtdyn1} to~\eqref{eq.IQC} yields the quadratic inequality,
\be
\tbo{\tilde{w}}{\tilde{u}}^T
\tbt{0}{C^T \Pi_0}{\Pi_0 C}{-2 \Lambda}
\tbo{\tilde{w}}{\tilde{u}}
\,\geq\, 0.
\label{eq.IQC1}
\ee
\end{subequations}
The sufficient condition for the global exponential stability of~\eqref{eq.dtdyn} is obtained by adding~\eqref{eq.IQC1} to~\eqref{eq.Vdot} and it amounts to the existence of a positive constant $\rho$ such that 
\be
\tbt{-(A^T P + P A + 2 \rho P)}{-(P B +  C^T \Pi_0)}{-(P B +  C^T \Pi_0)^T}{2  \Lambda}
\, \succ \, 0.
\label{eq.LMI0}
\ee
If this condition holds, we have $\dot{V} \leq -2 \rho V$. 
Thus, $V(\tilde{w}(t))\leq V(\tilde{w}(0)) \, \mathrm{e}^{-2\rho t}$ and since $P \succ 0$,  
	\[
	\Vert \tilde{w} (t) \Vert \, \leq \, \sqrt{\kappa_p} \, \Vert \tilde{w}(0) \Vert \, \mathrm{e}^{-\rho t},~\mbox{for all}~t \, \geq \, 0
	\]
where $\kappa_p$ is the condition number of the matrix $P$.
Since $\Lambda \succ 0$, the remaining task is to verify the existence of the positive parameters $\alpha$, $\mu$, $\lambda_1$, $\lambda_2$, and $\rho$ such that
	\be
	\ba{rcl}
	&&\!\!\!\!\!\!\!\!\!\! -(A^T P + P A + 2 \rho P) 
	\\[0.15cm]
	&& - ~
	\tfrac{1}{2} (P B + C^T \Pi_0) \, \Lambda^{-1} (P B + C^T \Pi_0)^T 
	\succ 0
	\ea
	\label{eq.key}
	\ee
which follows from the application of the Schur complement to~\eqref{eq.LMI0}. 

We are now ready to prove the global exponential stability of the primal-dual gradient flow dynamics~\eqref{eq.ahu} and provide estimates of the convergence rate $\rho$ for $L_f > m_f$. Similar result can be established for $L_f = m_f$.

\vsp
\begin{theorem}\label{thm.main1}
	Let Assumptions~\ref{as.feasible}-\ref{as.tt} hold, let $L_f > m_f$, and let $\sigma_{\max}(T)$ be the largest singular value of the matrix $T$. Then, the global exponential stability of the primal-dual gradient flow dynamics~\eqref{eq.ahu} can be established with Lyapunov function~\eqref{eq.lyap} if the augmented Lagrangian parameter satisfies	
	\be
	\mu > \max \left( \tfrac{L_f - m_f}{4}, \tfrac{\sigma^2_{\max}(T)}{8m_f}\left(1+\sqrt{1+ \tfrac{16m_f^2}{\sigma^2_{\max}(T)}} \right) \right).
	\label{eq.m}
	\ee
\end{theorem}
\begin{proof}
	If~\eqref{eq.key} holds for $\rho = 0$, the continuity of the left-hand side of~\eqref{eq.key} with respect to $\rho$ implies the existence of $\rho>0$ such that~\eqref{eq.key} holds. For $\rho=0$,~\eqref{eq.key} becomes
	\be
	-(A^T P + P A ) 
	\, - \,
	\tfrac{1}{2 } (P B + C^T \Pi_0) \, \Lambda^{-1} (P B + C^T \Pi_0)^T 
	\succ 0
	\label{eq.keys0}
	\ee
	where $A^T P + P A$ is given by~\eqref{eq.PAAP}, and $P B + C^T \Pi_0 $ reads
	\be
	\left[
	\begin{array}{cc}
		(\lambda_1 \hat{L}_1 - \alpha) \, I & \lambda_2 \, T^T 
		\\
		- (\alpha / \mu ) \, T & (\mu \lambda_2 -\alpha ( 1 +  {\alpha m_f} / {\mu} ) ) \, I
	\end{array}
	\right]
	\non
	\ee
where $\hat{L}_1 \DefinedAs L_f - m_f > 0$. Thus, the matrix $M \DefinedAs \tfrac{1}{2} (P B + C^T \Pi_0) \, \Lambda^{-1} (P B + C^T \Pi_0)^T$ is given by
		\be
	M \,=\,
	\left[
	\begin{array}{cc}
		M_1 & M_0^T
		\\
		M_0 & M_2
	\end{array}
	\right]
	\label{eq.M}
	\ee
where	 
	\be
	\begin{array}{rcl}
		M_1 
		&\!\!=\!\!& 
		\frac{1}{2} 
		\left(
		\frac{(\alpha - \lambda_1\hat{L}_1)^2}{\lambda_1} \, I  \, + \, {\lambda_2} \, T^T T 
		\right)
		\\[0.25cm]
		M_2 &\!\!=\!\!& 
		\frac{1}{2} 
		\left(
		\frac{\alpha^2}{\lambda_1 \mu^2} \, TT^T \, + \, \frac{1}{\lambda_2} ( \mu\lambda_2- \alpha(1+\frac{m_f}{\mu}) )^2 \,I
		\right)
		\\[0.25cm]
		M_0 &\!\!=\!\!& 
		\frac{1}{2} 
		\left(
		\frac{\alpha (\alpha - \lambda_1\hat{L}_1)}{\lambda_1 \mu} \, + \,  \mu \lambda_2 \, - \, \alpha (1+\frac{m_f}{\mu}) 
		\right) T
	\end{array}
	\non
	\ee
and $\alpha$, $\lambda_1$, $\lambda_2$, and $\mu$ are positive parameters that have to be selected such that~\eqref{eq.keys0} holds. Setting $\lambda_1 \DefinedAs \alpha/\hat{L}_1$ and $\lambda_2 \DefinedAs (\alpha/\mu) (1+m_f/\mu)$ yields $M_0=0$ and~\eqref{eq.keys0} simplifies~to
		\be
		\left[
		\begin{array}{cc}
			2\alpha m_f I -\frac{\lambda_2}{2} \, T^T T & 0
			\\
			0 & \frac{\alpha}{\mu} ( 2- \frac{\alpha}{2 \lambda_1 \mu} ) \, TT^T
		\end{array}
		\right] 	
		\, \succ \, 0
		\non 
		\ee
or, equivalently, 
	\[
	4\alpha m_f \, >  \, \lambda_2 \sigma^2_{\max}(T)
	~~\mbox{and}~~
	4\mu \lambda_1 \, > \, \alpha.
	\] 
Combining these two conditions with the above definitions of $\lambda_1$ and $\lambda_2$ yields~\eqref{eq.m}.
\end{proof}

\vsp

We next utilize the choices of parameters $\lambda_1$ and $\lambda_2$ in Theorem~\ref{thm.main1} to estimate the convergence rate $\rho$. 

	\vsp

\begin{proposition}\label{thm.main2}
	Let Assumptions~\ref{as.feasible}-\ref{as.tt} hold, let $L_f > m_f$, and let $\sigma_{\min}(T)$ and $\sigma_{\max}(T)$ be the smallest and the largest singular values of the matrix $T$. Then, the primal-dual gradient flow dynamics~\eqref{eq.ahu} are globally exponentially stable with the rate 
		\begin{subequations}
		\be
		\rho 
		\, \geq \, 
		\rho_0(\mu) 
		\, \DefinedAs \,
		\frac{\sigma_{\min}^2(T)}{2 (\mu \, + \, m_f \, + \, \sigma_{\max}^2(T)/\mu )}
		\label{eq.rho}
		\ee
if $ \mu > \max \, (L_f-m_f,\hat{\mu})$, where 		
		\be
		\hat{\mu}  
		\,\DefinedAs\, 
		\inf \left\{ 
		\mu\in[\sigma_{\max}(T),\infty), \; \beta(\mu)\, < \, 2 m_f  \right\}
		\label{eq.mu}
		\ee 
		\be
		\beta(\mu) 
		\, \DefinedAs \, 
		\frac{(m_f + \mu) \sigma_{\max}^2(T)}{2\mu^2} 
		\, + \, 
		\frac{2\rho_0(\mu) (\mu + 4\rho_0(\mu))}{\mu}.
		\label{eq.beta}
		\ee
		\end{subequations}
\end{proposition}

\vsp
	
\begin{proof}
See Appendix~\ref{sec.main-proof}. 
\end{proof}

	% \vspace*{-1ex}
\section{Computational experiments}\label{sec.comp}

We next provide an example to demonstrate the merits of our approach. Let us consider optimization problem~\eqref{eq.opt} with, \begin{equation}\label{eq.qpcp}
\begin{array}{rcl}
f(x) & \!\! = \!\! & \tfrac{1}{2} \, x^T Qx \; + \; q^T x
\\[0.15cm]
	g(z) 
	& \!\! = \!\! & 
 	\left\{
	\ba{rl}
	0, & z \, \leq \, b
	\\[0.1cm]
	\infty, & \mbox{otherwise} 
	\ea
	\right.
	\end{array}
	\end{equation}
where $x$ and $q$ are the $n$-dimensional vectors, $Q\in\mathbb{R}^{n\times n}$ is a positive definite matrix, $T\in\mathbb{R}^{m\times n}$ if a full row rank matrix, and $b \in \mathbb{R}^m$ is a given vector. The gradient of the Moreau envelope is determined by $\nabla M_{\mu g}(v_i) = \max \, (0, (v_i-b_i)/\mu)$ and ($L_f,m_f$) are the largest and the smallest eigenvalues of the matrix $Q$, respectively. 

\vsp
We use Matlab ODE solver $\mathsf{ode45}$ to simulate the primal-dual gradient flow dynamics~\eqref{eq.ahu} and set $n=m=10$, $q=10\times \mathsf{randn}(n,1)$, and $Q = HH^T + K$, where $H = \mathsf{randn}(n,n)$ and $K = \mathsf{diag(exp(randn(n,1)))}$. We choose $b$ to be a vector of all ones, set $T=I$, and report results for $(L_f,m_f) = (1.24,1.03)$ and $(L_f,m_f)=(27.81,1.03)$.

\vsp
Figure~\ref{fig.well} demonstrates the exponential convergence of dynamics~\eqref{eq.ahu} with $(L_f,m_f) = (1.24,1.03)$ for different values of $\mu$. We note that the convergence rate decreases when $\mu$ becomes larger than $2$. For a given value of $\mu$ that satisfies Proposition~\ref{thm.main2}, we use formula~\eqref{eq.rho} to estimate the lower bound on the convergence rate $\rho_0$. We compare our estimate with~\cite[Theorem~2]{quli19} and~\cite[Theorem~6]{dinjovACC19}. As shown in Fig.~\ref{fig.well.rho}, Proposition~\ref{thm.main2} provides a less conservative estimate of the convergence rate than the existing methods. As Figs.~\ref{fig.ill} and~\ref{fig.ill.rho} illustrates, similar observations can be made for a larger condition number, $(L_f,m_f)=(27.81,1.03)$. Clearly, the increase in condition number reduces the rate of exponential decay and our estimates are less conservative than those provided in the literature. 

\begin{figure}{}
	\centering
	\captionsetup[subfigure]{position=top}
	\begin{tabular}{ccc}
		\begin{tabular}{c}
		\rotatebox{90}{$\|w(t) \, - \, w^\star\|$ }
		\end{tabular}
		& 
		\hspace*{-0.75cm}
		\begin{tabular}{c}
			\subfloat[$(L_f,m_f) = (1.24,1.03)$]{\label{fig.well} \includegraphics[width=0.375\textwidth]{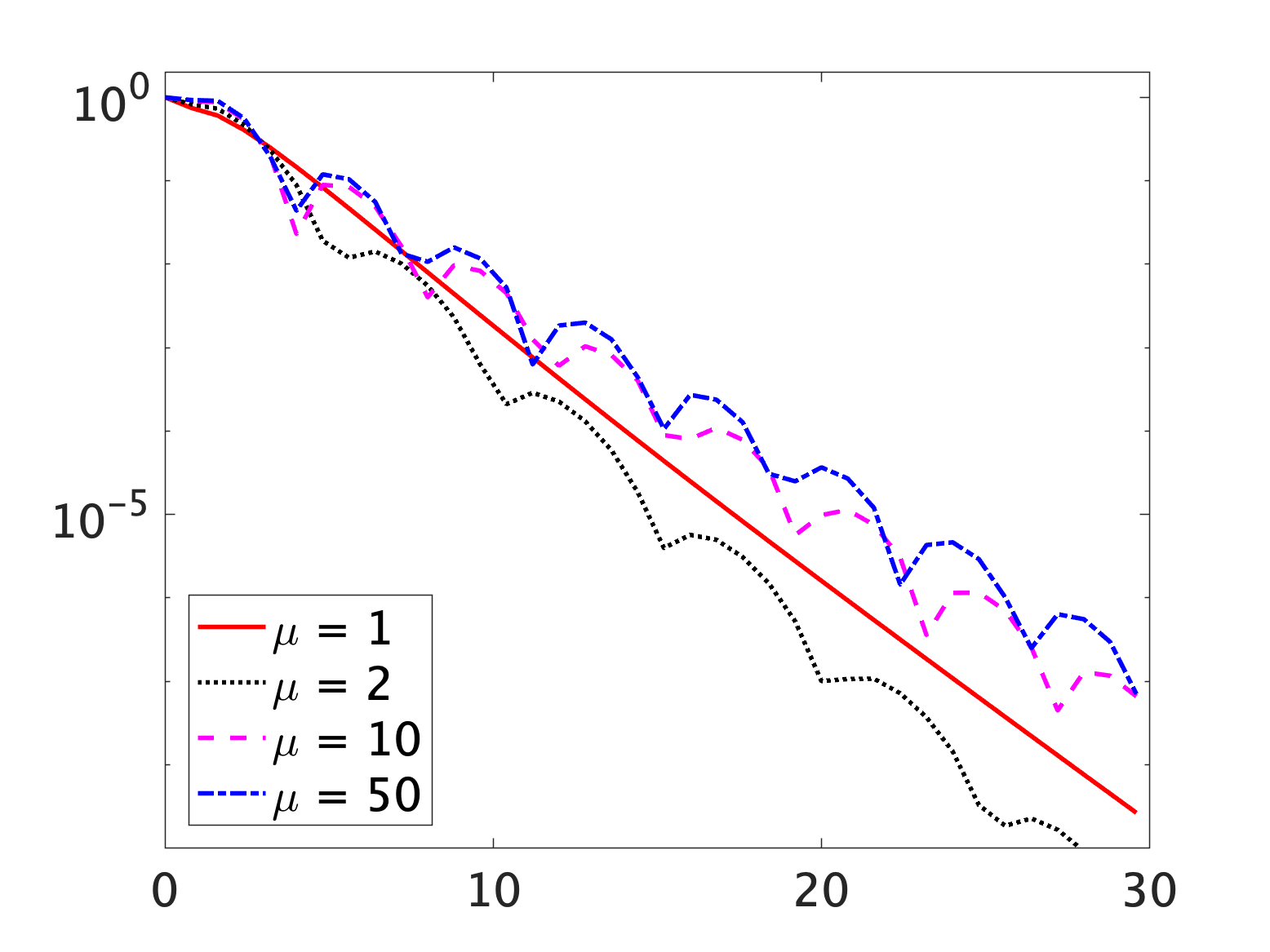}}
		\end{tabular}
		\\
		\begin{tabular}{c}
		\end{tabular} 
		& 
		\begin{tabular}{c}
			{time (seconds)} 
		\end{tabular}
		\\
		\begin{tabular}{c}
		\rotatebox{90}{$\|w(t) \, - \, w^\star\|$ }
		\end{tabular}
		&
		\hspace*{-0.75cm}
		\begin{tabular}{c}
			\subfloat[$(L_f,m_f)=(27.81,1.03)$]{\label{fig.ill} \includegraphics[width=0.375\textwidth]{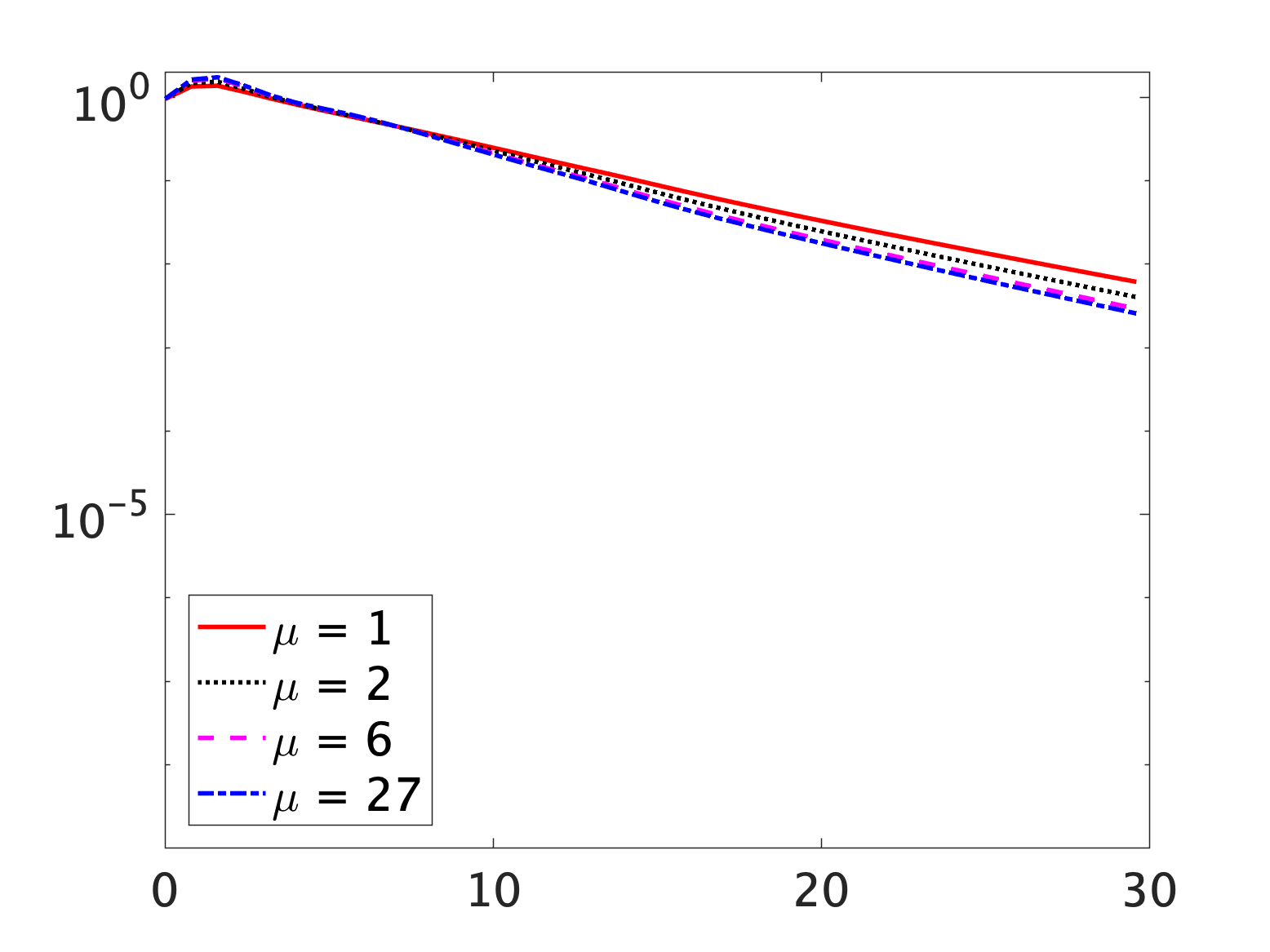}}
		\end{tabular}
		\\
		\begin{tabular}{c}
		\end{tabular} 
		& 
		\begin{tabular}{c}
			{time (seconds)} 
		\end{tabular}
	\end{tabular}
	\caption{Convergence of the primal-dual gradient flow dynamics~\eqref{eq.ahu} for problem~\eqref{eq.qpcp} with (a) $(L_f,m_f) = (1.24,1.03)$ and (b) $(L_f,m_f)=(27.81,1.03)$.}
	\label{AHU_EDP}
\end{figure}

\begin{figure}{}
	\centering
	\captionsetup[subfigure]{position=top}
	\begin{tabular}{ccc}
		\begin{tabular}{c}
		\rotatebox{90}{$\rho$ }
		\end{tabular}
		& 
		\hspace*{-0.75cm}
		\begin{tabular}{c}
				\subfloat[$(L_f,m_f) = (1.24,1.03)$]{\label{fig.well.rho} 			
				\includegraphics[width=0.375\textwidth]{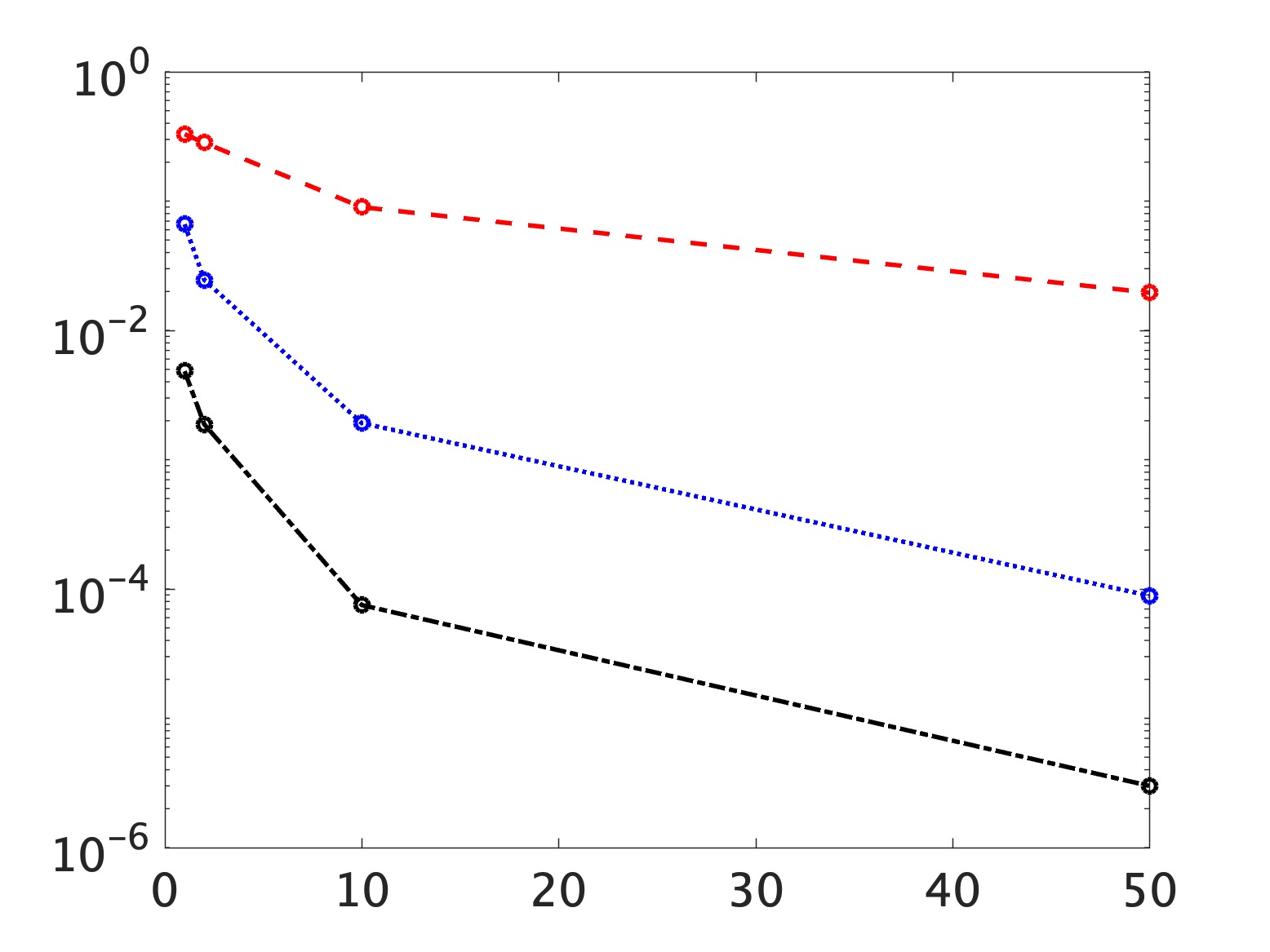}}
		\end{tabular}
		\\
		\begin{tabular}{c}
		\end{tabular} 
		& 
		\begin{tabular}{c}
			{$\mu$} 
		\end{tabular}
		\\
		\begin{tabular}{c}
		\rotatebox{90}{$\rho$ }
		\end{tabular}
		&
				\hspace*{-0.75cm}
		\begin{tabular}{c}
			\subfloat[$(L_f,m_f)=(27.81,1.03)$]{\label{fig.ill.rho} 
			\includegraphics[width=0.375\textwidth]{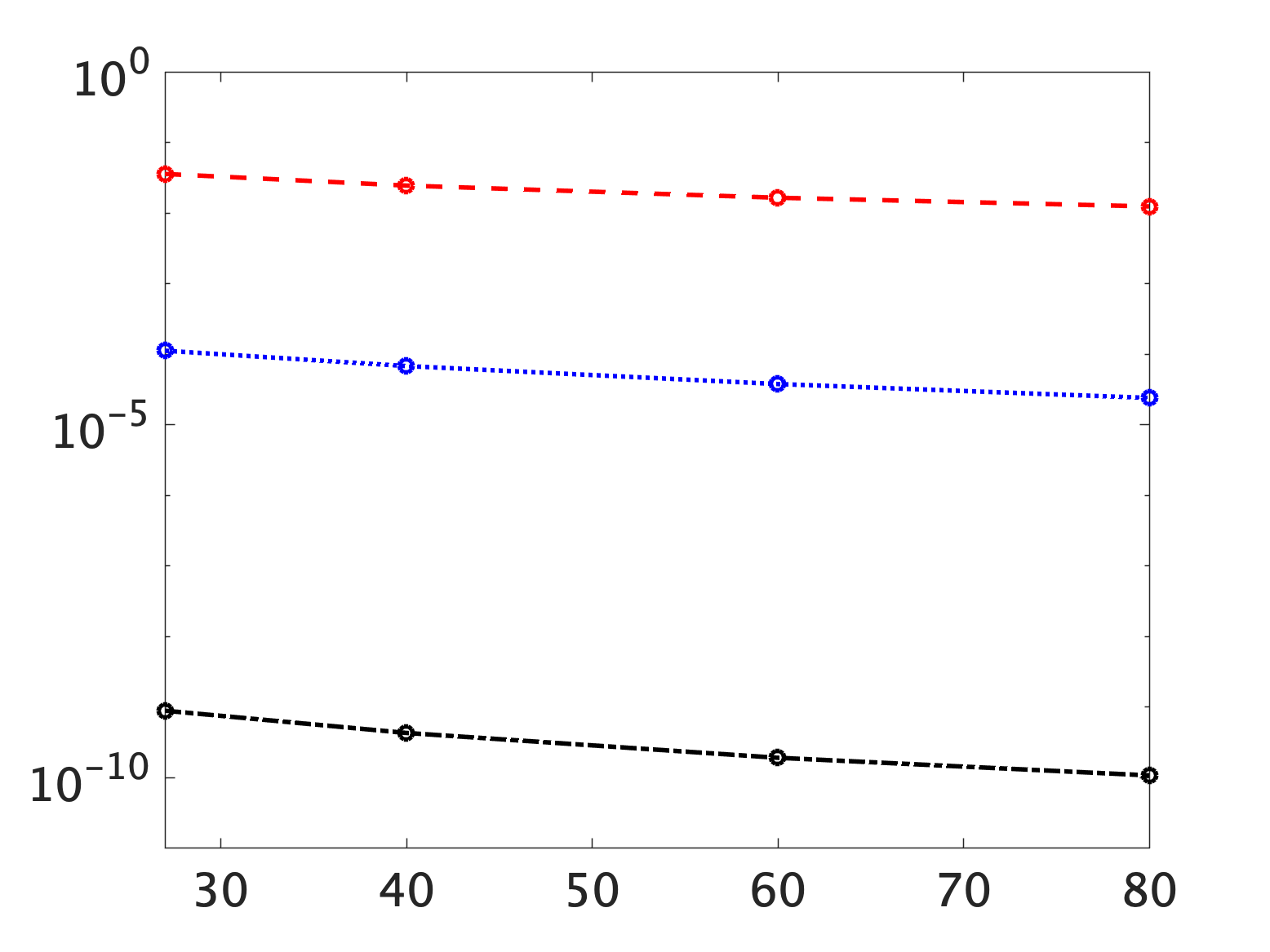}}
		\end{tabular}
		\\
		\begin{tabular}{c}
		\end{tabular} 
		& 
		\begin{tabular}{c}
			{$\mu$} 
		\end{tabular}
	\end{tabular}
	\caption{Convergence rate estimates, as a function of $\mu$, resulting from~\eqref{eq.rho}~(\textbf{\color{red}-- \!--}),~\cite[Theorem~6]{dinjovACC19} (\textbf{\color{blue}$\bm\cdot$$\bm\cdot$$\bm\cdot$}), and~\cite[Theorem~2]{quli19} (\textbf{-- \!-}) for problem~\eqref{eq.opt} with~\eqref{eq.qpcp} and (a) $(L_f,m_f) = (1.24,1.03)$;~(b) $(L_f,m_f)=(27.81,1.03)$. }
	\label{AHU_EDPT}
\end{figure}

	%\vspace*{-1ex}
\section{Concluding remarks}\label{sec.con}

In this paper, we use a Lyapunov-based approach to establish global exponential stability of the primal-dual gradient flow dynamics resulting from the proximal augmented Lagrangian framework for nonsmooth composite optimization. We provide a worst-case estimate of the exponential decay rate when the differentiable part of the objective function is strongly convex and its gradient is Lipschitz continuous. For a quadratic programming problem, computational experiments are used to show that our estimate of the convergence rate is less conservative compared to the existing literature. Our ongoing work focuses on identifying a quadratic Lyapunov function that can certify the global exponential stability of a second-order primal-dual method for nonsmooth composite optimization~\cite{dhikhojovTAC17}.

	\vspace*{-1ex}
\appendix

	\vspace*{-1ex}
\subsection{Proof of Proposition~\ref{thm.main2}}
	\label{sec.main-proof}
	
	We show that~\eqref{eq.key} holds for $\rho=\rho_0(\mu)$. Substitution of the expressions for $A^T P+PA$ and $\tfrac{1}{2}(PB+C^T \Pi_0) \Lambda^{-1} (B^TP+\Pi_0 C)$ given by~\eqref{eq.PAAP} and~\eqref{eq.M} into~\eqref{eq.key} yields 
\be
R \,=\,
\left[
\begin{array}{cc}
	R_1 & R_0^T
	\\
	R_0 & R_2
\end{array}
\right] \,\succ\, 0
\label{eq.Q}
\ee 
where
\be
\ba{rcl}
	R_1 &\!\!=\!\!& 2\alpha(m_f-\rho) \,I\, -\, \frac{(\alpha-\lambda_1 \hat{L}_1)^2}{2 \lambda_1 } \,I\, -\, \frac{\lambda_2}{2} \,T^T T
	\\[0.15cm]
	R_2 &\!\!=\!\!& (\frac{2\alpha}{\mu}-\frac{\alpha^2}{2\lambda_1\mu^2})\,TT^T \,-\, 2\rho\alpha ((1+\frac{m_f}{\mu})\,I+\frac{1}{\mu^2} \,TT^T) 
	\\[0.15cm]
	&& \,-\, \frac{1}{2\lambda_2} ( \mu\lambda_2- \alpha(1+\frac{m_f}{\mu}) )^2 \,I
	\\[0.15cm]
	R_0 &\!\!=\!\!& \,-\,\frac{\alpha}{\mu} ( \frac{\alpha-\lambda_1\hat{L}_1}{2\lambda_1}+ 2\rho ) \,T \,-\,  \frac{1}{2}( \mu \lambda_2 - \alpha (1+\frac{m_f}{\mu}) )\,T.
\ea
\non
\ee
Here, $\hat{L}_1 \DefinedAs L_f-m_f>0$, and $\alpha$, $\lambda_1$, $\lambda_2$, and $\mu$ are positive parameters that have to be selected such that~\eqref{eq.key} holds for $\rho=\rho_0(\mu)$. We set $\lambda_1 \DefinedAs \alpha/\hat{L}_1$, $\lambda_2 \DefinedAs \alpha(1+m_f/\mu)/\mu$, and add/subtract $3\alpha TT^T/ (2\mu)$ to $R_2$ to obtain 
	\be
	\ba{rcl}
	R_2 & \!\!=\!\! & \alpha(\tfrac{2}{\mu}-\tfrac{\hat{L}_1}{2\mu^2})\,TT^T \,-\, \tfrac{3\alpha}{2\mu} \,TT^T  ~+
	\\[0.15cm]
	& \!\! \!\! &
	\tfrac{\alpha}{\mu} \,TT^T\,-\,2\rho\alpha ((1+\tfrac{m_f}{\mu})\,I+\tfrac{1}{\mu^2} \,TT^T) \,+\, \tfrac{\alpha}{2\mu}\,TT^T.
	\ea
	\non
	\ee
	If $\mu\geq \hat{L}_1$, then
	\be
	\alpha(\tfrac{2}{\mu}-\tfrac{\hat{L}_1}{2\mu^2})\,TT^T\, -\, \tfrac{3\alpha}{2\mu} \,TT^T \;\succeq\;0.
	\label{eq.lgeq1}
	\ee
Furthermore, for $\rho=\rho_0$, we have
	\be
	\ba{l}
	\!\!\!\!
	\tfrac{\alpha}{\mu}\, TT^T\,-\,2\rho\alpha((1+\tfrac{m_f}{\mu})\,I+\tfrac{1}{\mu^2}\, TT^T) 
	~ \succeq
	\\[0.15cm]
	\!\!\!\!
	\tfrac{\alpha}{\mu} \, \sigma_{\min}^2(T) \,I \,-\,2\rho \alpha(1+\tfrac{m_f}{\mu}+\tfrac{\sigma_{\max}^2(T)}{\mu^2} )\, I
	\, = \, 0.
	\ea
	\label{eq.lgeq2}
	\ee
Combining~\eqref{eq.lgeq1} and~\eqref{eq.lgeq2} with the definition of $R_2$ yields $R_2\succeq\alpha TT^T/(2\mu)$ and the positive definiteness of $R_2$ follows from the fact that $T$ is a full row rank matrix.
	
	 The application of the Schur complement requires $R_1 - R_0^T R_2^{-1} R_0 \succ 0$. Using $R_2\succeq\alpha TT^T/(2\mu)$, we can rewrite this condition as 
	$
	R_1 - ( 2\mu/\alpha) R_0^T (TT^T)^{-1} R_0 \succ0.
	$
For $\beta(\mu)$ and $\hat{\mu}$ given by~\eqref{eq.beta} and~\eqref{eq.mu}, respectively, if $\mu> \hat{\mu}$, we have
	\[
	\begin{array}{rcl}
	&&\!\!\!\!\!\!\!\!\!\!\!\!\!\!\!\!\!\!\!\!\!\!\!\!\! R_1 - (2\mu/\alpha) \,R_0^T\, (TT^T)^{-1} \,R_0
	\\[0.1cm]
	& \!\!\succeq\!\! & \alpha \,\big(2m_f  -( 2\rho_0+ \frac{\sigma_{\max}^2(T)}{2\mu} (1+\frac{m_f}{\mu}) +\frac{8\rho_0^2}{\mu} )\big) \,I
	\\[0.1cm]
	& \!\!=\!\! &  \alpha\, \big(2m_f  - \beta(\mu)\big) \,I
	\; \succ \; 0.
	\end{array}
	\]
We now prove the existence of such $\hat{\mu}$. Since $\rho_0(\mu)$ is monotonically decreasing for $\mu\geq \sigma_{\max}(T)$, $\beta(\mu)$ monotonically decreases to zero on the interval $\mu \in [\sigma_{\max}(T),\infty)$. There are two cases: (i) if $\beta(\sigma_{\max}(T)) > 2 m_f$, then $\beta(\bar{\mu}) = 2m_f$ for some $\bar{\mu} \in [\sigma_{\max}(T),\infty)$ and $\hat{\mu} = \bar{\mu}$; (ii) if $\beta(\sigma_{\max}(T)) \leq 2 m_f$, then $\beta(\mu) < \beta(\sigma_{\max (T)}) \leq 2 m_f$ for all $\mu \in (\sigma_{\max (T)}, \infty)$. Thus $\hat{\mu} = \sigma_{\max (T)}$. Therefore, the set in~\eqref{eq.mu} is nonempty and such $\hat{\mu}$ always exists.

% references

\end{document}